\documentclass[11 pt]{article}
\usepackage{amssymb,amsmath,comment}
\catcode`\@=11 \@addtoreset{equation}{section}
\def\thesection{\arabic{section}}

\def\theequation{\thesection.\arabic{equation}}
\catcode`\@=12

\usepackage{epsfig,enumerate,amsmath,amsfonts,amssymb}
\usepackage{indentfirst}
\usepackage{setspace}
\usepackage{latexsym}
\usepackage{bm}
\usepackage{bbm}
\usepackage{amsthm}
\usepackage[english,french]{babel} 
\usepackage{colortbl}%
\usepackage{a4wide}
\usepackage{parskip}

\usepackage[hidelinks]{hyperref}
\newcommand{\tl}{\tilde}

\newcommand{\e}{\epsilon}
\newcommand{\vep}{\varepsilon}
\newcommand{\pa} {\partial}
\newcommand{\al} {\alpha}
\newcommand{\ba} {\beta}
\newcommand{\de} {\delta}

\newcommand{\Ga} {\Gamma}
\newcommand{\Om} {\Omega}
\newcommand{\sg}{\sigma}
\newcommand{\ra} {\rightarrow}
\newcommand{\ov}{\overline}
\newcommand{\De} {\Delta}

\newcommand{\ka}{\kappa}
\newcommand{\noi} {\noindent}
\newcommand{\na} {\nabla}

\newcommand{\mb} {\mathbb}
\newcommand{\mc} {\mathcal}

\usepackage[all]{xy}
\catcode`\@=11
\def\theequation{\@arabic{\c@section}.\@arabic{\c@equation}}
\catcode`\@=12

\def\QED{\hfill {$\square$}\goodbreak \medskip}

\newtheorem{Theorem}{Theorem}[section]

\newtheorem{Proposition}[Theorem]{Proposition}
\newtheorem{Corollary}[Theorem]{Corollary}
\newtheorem{Remark}{Remark}
\newtheorem{Definition}{Definition}

\def\Xint#1{\mathchoice
	{\XXint\displaystyle\textstyle{#1}}%
	{\XXint\textstyle\scriptstyle{#1}}%
	{\XXint\scriptstyle\scriptscriptstyle{#1}}%
	{\XXint\scriptscriptstyle\scriptscriptstyle{#1}}%
	\!\int}
\def\XXint#1#2#3{{\setbox0=\hbox{$#1{#2#3}{\int}$ }
		\vcenter{\hbox{$#2#3$ }}\kern-.6\wd0}}

\newenvironment{poliabstract}[1]
   {\renewcommand{\abstractname}{#1}\begin{abstract}}
   {\end{abstract}}

\begin{document}
	\title
	{A note on the global regularity results for strongly nonhomogeneous $p,q$-fractional problems and applications}
	 
	\author{  Jacques Giacomoni$^{\,1}$ \footnote{e-mail: {\tt jacques.giacomoni@univ-pau.fr}}, \ Deepak Kumar$^{\,2}$\footnote{e-mail: {\tt deepak.kr0894@gmail.com}},  \
		and \  Konijeti Sreenadh$^{\,2}$\footnote{
			e-mail: {\tt sreenadh@maths.iitd.ac.in}} \\
		\\ $^1\,${\small Universit\'e  de Pau et des Pays de l'Adour, LMAP (UMR E2S-UPPA CNRS 5142) }\\ {\small Bat. IPRA, Avenue de l'Universit\'e F-64013 Pau, France}\\  
		$^2\,${\small Department of Mathematics, Indian Institute of Technology Delhi,}\\
		{\small	Hauz Khaz, New Delhi-110016, India} }

	\date{}
	
	\maketitle
 \selectlanguage{English}
\begin{poliabstract}{Abstract}
In this article, we communicate with the glimpse of the proofs of new global regularity results for weak solutions to a class of problems involving fractional $(p,q)$-Laplacian, denoted by $(-\Delta)^{s_1}_{p}+(-\Delta)^{s_2}_{q}$, for $s_2, s_1\in (0,1)$ and $1<p,q<\infty$. We also obtain the boundary H\"older continuity results for the weak solutions to the corresponding problems involving at most critical growth nonlinearities. These results are almost optimal. 
Moreover, we establish Hopf type maximum principle and strong comparison principle. 
As an application to these new results, we prove the Sobolev versus H\"older minimizer type result, which provides the multiplicity of solutions in the spirit of seminal work \cite{Brezis-Nirenberg}.

\end{poliabstract}

\selectlanguage{French}
\begin{poliabstract}{R\'esum\'e}
Dans cette note, nous présentons de nouveaux résultats de régularité Höldérienne des solutions faibles d'une classe de problèmes faisant intervenir des opérateurs de diffusion fractionnaire non linéaires et non homogènes de la forme $(-\Delta)^{s_1}_{p}+(-\Delta)^{s_2}_{q}$ avec $s_2, s_1\in (0,1)$ et $1<p,q<\infty$. Précisément, nous obtenons des résultats de régularité intérieure et près du bord pour les solutions faibles de ces problèmes alors que la nonlinéarité du membre de droite est de croissance critique au sens de l'injection de Sobolev. Ce résultat étend les principaux résultats de régularité intérieure de \cite{brascoH} où le cas de l'opérateur homogène $(-\Delta)^{s_1}_{p}$ est investi, améliore de façon optimale et complète ceux de \cite{DDS}. 

Nous établissons par ailleurs un lemme de Hopf et un principe de comparaison fort pour cette classe de problèmes. Nous appliquons ensuite ces résultats pour démontrer la propriété que les minimiseurs locaux de l'énergie associée dans $C^\al(\overline{\Omega})$ avec $\al\in (0,s_1)$ sont aussi minimiseurs locaux dans $W^{s_1,p}_0(\Omega)$ dans l'esprit de l'article pionnier \cite{Brezis-Nirenberg}. Ceci conduit à des nouveaux résultats de muliplicité de solutions pour ces problèmes non locaux et fortement non homogènes.
\end{poliabstract}
\noi \textbf{Keywords:} Fractional $(p,q)$-Laplacian, non-homogeneous nonlocal operator, 
 H\"older continuity up to the boundary, maximum principle, strong comparison principle.
\\
\noi \textit{2010 Mathematics Subject Classification:}  35J60, 35R11, 35B45, 35D30.
\section {Introduction}
 In this note we study the H\"older continuity results and maximum principle for weak solutions to the following problem: 
 \begin{equation*}
 \begin{array}{rllll}
 (-\Delta)^{s_1}_{p}u+(-\Delta)^{s_2}_{q}u  = f \quad \text{in } \Om,
 \end{array}
 \tag{$\mc P$}\label{probM}
 \end{equation*}
 where $\Om\subset\mb R^N$ is a bounded domain with $C^{1,1}$ boundary, $2\leq q,p<\infty$,  $0<s_2\leq s_1<1$  and $f\in L^\infty_{\mathrm{loc}}(\Om)$.  The fractional $p$-Laplace operator $(-\Delta)^{s}_{p}$ is defined as
 \begin{equation*}
 	{(-\Delta)^{s}_pu(x)}= 2\lim_{\vep\to 0}\int_{\mathbb R^N\setminus B_{\vep}(x)} \frac{|u(x)-u(y)|^{p-2}(u(x)-u(y))}{|x-y|^{N+ps}}dy.
 \end{equation*} 
 The leading differential operator, $(-\Delta)^{s_1}_{p}+(-\Delta)^{s_2}_{q}$, in problem \eqref{probM} is known as the fractional $(p,q)$-Laplacian. The operator is non-homogeneous in the sense that for any $t>0$, there does not exist any $\sg\in\mb R$ such that $((-\Delta)^{s_1}_{p}+(-\Delta)^{s_2}_{q})(tu)=t^\sg ((-\Delta)^{s_1}_{p}u+(-\Delta)^{s_2}_{q}u)$ holds for all $u\in W^{s_1,p}(\Om)\cap W^{s_2,q}(\Om)$.\par
 The regularity results and maximum principles for the equations involving the homogeneous nonlocal operators are well known, see  \cite{brascoH,delpezzo,iann,ianndist,ros}, whereas the regularity issues for the problems involving the fractional $(p,q)$-Laplacian is still developing and only few continuity results are available, see for instance, \cite{JDS2,JDSjga}.
The strong non-homogeneity, in this case, feature creates an additional difficulty while handling the distance function in order to prove the boundary behavior of the weak solutions.\par 
In the present work, we obtain interior regularity results  for local weak solutions, which improves the regularity results of \cite{JDS2} for larger class of exponents.   Our proof of the improved local H\"older continuity result (see Theorem \ref{impintreg}) relies on Moser's iteration technique to obtain suitable embedding for Besov spaces into the H\"older spaces. Here, we stress that we do not assume any order relation on the exponents $p$ and $q$. Subsequently, we establish the asymptotic behavior of the fractional $q$-Laplacian ($(-\De)_q^{s_2}$) of the distance function $d^{s_1}$ near the boundary, which in turn gives us almost optimal (and optimal in some cases, see Remark \ref{rem1opt}) boundary behavior of the weak solution. 
This coupled with the interior H\"older regularity result of Theorem \ref{impintreg} provides the almost optimal H\"older continuity result. As a consequence of this, we obtain the Hopf type maximum principle for non-negative solutions. Additionally, under the restriction that the fractional $q$-Laplacian of the subsolution is bounded from below, we prove a strong comparison principle. In the last section, as an application to these results, we obtain the multiplicity results for problem involving the subcritical nonlinearity by establishing Sobolev versus H\"older type minimization result for nonlinearities with atmost critical growth.
Complete proofs of the regularity main results and other applications (in particular to singular problems) can be found in \cite{JDS2}.

\section{Preliminaries and main results}
 We denote $[t]^{p-1}:=|t|^{p-2}t$, for all $p>1$ and $t\in\mb R$. For $(\ell,s)\in\{(p,s_1),(q,s_2)\}$ and for $S_1\times S_2\subset\mb R^{N}\times\mb R^N$, we set
 \begin{align*}
 	A_\ell(u,v,S_1\times S_2)= \int_{S_1\times S_2}\frac{[u(x)-u(y)]^{\ell-1}(v(x)-v(y))}{|x-y|^{N+\ell s}}~dxdy.
 \end{align*}
 We define the distance function as $d(x):=\mathrm{dist}(x,\mb R^N\setminus\Omega)$ and a neighborhood of the boundary as $\Om_{\varrho}:=\{ x\in \Om \ : \ d(x)<\varrho \}$,  for $\varrho>0$. \\ We will follow the notation $p^*_{s_1}:=Np/(N-ps_1)$ if $N>ps_1$, otherwise an arbitrarily large number.
 \subsection{Function Spaces}
 For $E\subset\mathbb{R}^N$, $p\in [1,\infty)$ and $s\in (0,1)$, the fractional Sobolev space $W^{s,p}(E)$ is defined as 
 \begin{align*}
 	W^{s,p}(E):= \left\lbrace u \in L^p(E): [u]_{W^{s,p}(E)}  < \infty \right\rbrace
 \end{align*}
 \noi endowed with the norm $\|u\|_{W^{s,p}(E)}:=  \|u\|_{L^p(E)}+ [u]_{W^{s,p}(E)}$,
 where \begin{align*}
 	[u]_{W^{s,p}(E)}:=  \left(\int_{E}\int_{E} \frac{|u(x)-u(y)|^p}{|x-y|^{N+sp}}~dxdy  \right)^{1/p}.
 \end{align*} 
For any (proper) subset $E$ of $\mb R^N$, we have 
 \begin{align*}
 	W^{s,p}_0(E):=\{ u\in W^{s,p}(\mb R^N) \ : \ u=0 \quad\mbox{in }\mb R^N\setminus E \}
 \end{align*}
 which is a uniformly convex Banach space when equipped with the norm  $[\cdot]_{W^{s,p}(\mb R^N)}$ (hereafter, it will be denoted by $\|\cdot\|_{W^{s,p}_0(E)}$). 
 Next, we define 
 $$\mc W(E):= W^{s_1,p}(E)\cap W^{s_2,q}(E)$$
 equipped with the norm $\|\cdot\|_{\mc W(E)}:=\|\cdot\|_{W^{s_1,p}(E)}+\|\cdot\|_{W^{s_2,q}(E)}$. The space $\mc W_0(E)$ is defined analogously. We say that $u\in \mc W_{\rm loc}(E)$ if  $u\in\mc W(E')$, for all $E'\Subset E$.
  Note that for $1<q\leq p<\infty$, $0<s_2<s_1<1$ and the domain $E$ with Lipschitz boundary, $W^{s_1,p}_0(E)$ coincides with  the space $X_{p,s_1}$, as defined in \cite{DDS}. Indeed, from \cite[Lemma 2.1]{DDS}, we have 
  \begin{align*}
  	\|u\|_{W^{s_2,q}_0(E)}\leq C \|u\|_{W^{s_1,p}_0(E)}, \quad \text{for all } \; u \in W^{s_1,p}_0(E),
  \end{align*}
 for some $C=C(|E|,\;N,\; p,\;q,\;s_1,\;s_2)>0$. Additionally, we define
\begin{align*}
 	\widetilde{W}^{s,p}(\Om):= \bigg\{u\in L^{p}_{\rm loc}(\mb R^N) \ : \ \exists\Om'\Supset\Om \mbox{ s.t. } u\in W^{s,p}(\Om'), \ \int_{\mb R^N}\frac{|u(x)|^{p-1}}{(1+|x|)^{N+ps}}dx<\infty  \bigg \}.
 \end{align*}
 \begin{Definition}
 Let $u:\mb R^N\to \mb R$ be a measurable function and $0<m,\alpha<\infty$. We define the tail space and the nonlocal tail, respectively, as below:
 	{\small\begin{align*}
 	  L^{m}_{\al}(\mb R^N) = \bigg\{ u\in L^{m}_{\rm loc}(\mb R^N) : \int_{\mb R^N} \frac{|u(x)|^{m}dx}{(1+|x|)^{N+\al}} <\infty \bigg\}, \ \  T_{m,\al}(u;x_0,R)=\left(R^{\al}\int_{B_R(x_0)^c} \frac{|u(y)|^{m}dy}{|x_0-y|^{N+\al}} \right)^\frac{1}{m}.
 	\end{align*}} 
 Set $T_{m,\al}(u;R)=T_{m,\al}(u;0,R)$. We will follow the notation $T_{p-1}(u;x,R):=T_{p-1,s_1p}(u;x,R)$ and $T_{q-1}(u;x,R):=T_{q-1,s_2q}(u;x,R)$, unless otherwise stated.
 \end{Definition}

 \subsection{Statements of main results}
 In this subsection, we state our main results. We start with the definition of local weak solution.
 \begin{Definition}[Local weak solution]
 A function $u\in \mc W_{\rm loc}(\Om)\cap L^{p-1}_{s_1p}(\mb R^N) \cap L^{q-1}_{s_2q}(\mb R^N)$ is said to be a local weak solution of problem \eqref{probM} if 
 	\begin{align*}
 		A_p(u,\phi,\mb R^N\times \mb R^N) + A_q (u, \phi,\mb R^N\times \mb R^N)= \int_{\Om} f\phi dx,
 	\end{align*}
 	for all $\phi\in \mc W(\Om)$ with compact support contained in $\Om$. 
 \end{Definition}
 Our first main theorem is the following higher local H\"older continuity result.
 \begin{Theorem}\label{impintreg}
 Suppose that $(q-p+2)s_2<2$. Let $u\in \mc W_{\rm loc}(\Om)\cap L^{p-1}_{s_1p}(\mb R^N) \cap  L^{q-1}_{s_2q}(\mb R^N)$ be a locally bounded local weak solution to problem \eqref{probM}. Then, for every $\sg\in (0,\Theta)$, $u\in C^{0,\sg}_{\mathrm{loc}}(\Om)$, where
 	\begin{align*}
 		\Theta\equiv\Theta(p,s_1,q,s_2)=\begin{cases}
 			\min\{1, ps_1/(p-1)\} \mbox{ if }qs_2<ps_1+2(1-s_1),\\
 			\min\{1, qs_2/(q-1)\} \mbox{ if }ps_1<qs_2.
 		\end{cases}
 	\end{align*} 
 Moreover, for $B_{2\bar R_0}(x_0)\Subset\Om$ with $\bar R_0\in (0,1)$, there holds
 	\begin{align*}
 		[u]_{C^{\sg}(B_{\bar R_0/2}(x_0))}\leq C \big(K_2(u) ( \|u\|_{W^{s_1,p}(B_{\bar R_0}(x_0))} + 1) \big)^{j_\infty}
 	\end{align*}
 where  $C=C(N,s_1,p,s_2,q,\sg)>0$ is a constant and $K_2$ is given by
 	{\small\begin{align*}
 			K_2=1+T_{p-1}(u;x_0,\bar R_0)^{p-1} +T_{q-1}(u;x_0,\bar R_0)^{q-1} 
 			+\|u\|^{\frac{(\ell_1+j_\infty)(\ell_1-1)}{\ell_1-2}}_{L^\infty(B_{\bar R_0}(x_0))}+\| u \|^{q-1}_{L^\infty(B_{R_0})} +\|f\|_{L^\infty(B_{\bar R_0}(x_0))}
 	\end{align*}} 
 with $\ell_1=\max\{p,q\}$ and $j_\infty\in\mb N$ depends only on $N,\sg$ and $(p,s_1)$ or $(q,s_2)$. 
 \end{Theorem}
\begin{Remark}
 We remark that for $1<q\leq p<\infty$, the conclusion of Theorem \ref{impintreg} holds for some $\sg<\min\{\frac{ps_1}{p-1},\frac{qs_2}{q-1}\}$. See Theorem 2.1 and Corollary 2.1 of \cite{JDSjga} for details.
\end{Remark}
Next, we have the following global H\"older continuity result.
\begin{Theorem}\label{bdryreg}
  Suppose that $(q-p+2)s_2<2$. Let $u\in \mc W_0(\Om)$ be a solution to problem \eqref{probM} with $f\in L^\infty(\Om)$. Then, for every $\sg\in(0,s_1)$, $u\in C^{0,\sg}(\ov\Om)$. Moreover, 
 \begin{align}\label{holderbd}
 	\| u \|_{C^\sg(\ov\Om)} \leq C,
 \end{align}
 where $C=C(\Om,N,p,s_1,q,s_2,\sg,\|f\|_{L^\infty(\Om)})>0$ is a constant (which depends as a non-decreasing function of $\|f\|_{L^\infty(\Om)}$). 
\end{Theorem}
 
\begin{Corollary}\label{cor1}
 Suppose that $2\leq q\leq p<\infty$. Let $u\in W^{s_1,p}_0(\Om)$ be a solution to problem \eqref{probM} with  $f(x):=f(x,u)$, a Carath\'eodory function satisfying $|f(x,t)|\leq C_0 (1+|t|^{p^*_{s_1}-1})$, where $C_0 >0$ is a constant. Then, $u\in C^{0,\sg}(\ov\Om)$, for all $\sg\in (0,s_1)$, and \eqref{holderbd} holds. 
\end{Corollary}
 Now, we mention our strong comparison theorem.
 \begin{Theorem}[Strong Comparison principle]\label{strongcompreg}
 Suppose that $1<q\leq p<\infty$. Let $u,v\in W^{s_1,p}_0(\Om)\cap C(\ov\Om)$ be such that $0<v\leq u$ in $\Om$ with $u\not\equiv v$, and for some $K, K_1>0$, the following holds:
	\begin{equation*}
	 (-\De)_p^{s_1}v +(-\De)_q^{s_2}v \leq (-\De)_p^{s_1}u +(-\De)_q^{s_2}u \leq K  \quad\mbox{and }
		(-\De)_q^{s_2} v \geq -K_1,
		\ \ \ \mbox{weakly in }\Om.
	\end{equation*}
 Then $u>v$ in $\Om$. Moreover, for $s_1\neq q's_2$, $\frac{u-v}{d^{s_1}}\geq C>0$ in $\Om$.
\end{Theorem}

 \section{H\"older regularity results}
 We first recall some boundedness results. 
 \begin{Proposition}[Local boundedness]
 Suppose $1<q\leq p<\infty$.  Let $u\in W^{s_1,p}_{\rm loc}(\Om)\cap L^{p-1}_{s_1p}(\mb R^N) \cap L^{q-1}_{s_2q}(\mb R^N)$ be a local weak solution to the problem \eqref{probM}. Then, $u\in L^\infty_{\mathrm{loc}}(\Om)$, and the following holds
 \begin{align*}
 	\|u\|_{L^\infty(B_{r/2}(x_0))}  \leq C \left( \Xint-_{B_r}|u|^p dx \right)^{1/p}+ T_{p-1}(u;x_0,\frac{r}{2})+T_{q-1}(u;x_0,\frac{r}{2})^\frac{q-1}{p-1}+ \|f\|_{L^\infty(B_r)}^{1/(p-1)}+1,
 \end{align*}
 where $C(N,p,q,s_1)>0$ is a constant. 
\end{Proposition}
 Proceeding similar to \cite[Theorem 3.3]{chenJFA} and noticing that the terms corresponding to the fractional $q$-Laplacian will be non-negative (using the similar inequality, as in the fractional $p$-Laplacian, for different $G_\ba$ in there), we can prove the following boundedness property.
\begin{Theorem}\label{mainbound}
 Let $1<q\leq p<\infty$.  Let $u\in W^{s_1,p}_0(\Om)$ be a weak solution to problem \eqref{probM} with $f(x):=f(x,u)$ satisfying $|f(x,t)|\leq C_0(1+|t|^{p^*_{s_1}-1})$, for all $t\in\mb R$ and a.e. $x\in\Om$, where $C_0>0$ is a constant. Then, $u\in L^\infty(\Om)$. 
\end{Theorem}	
\begin{Remark}\label{rem2}
 We remark that, as in \cite[Remark 3.4]{chenJFA}, the quantity $\|u \|_{L^\infty(\Om)}$ depends only on the constants $C_0$, $N$, $p$, $s_1$, $\|u\|_{W^{s_1,p}_0(\Om)}$ and the constant $M>0$ satisfying $\int_{\{|u|\geq M\}} |u|^{p^*_{s_1}}<\e$, for given $\e\in (0,1)$.
\end{Remark}
\begin{Corollary}\label{bound}
 Suppose that $1<q\leq p<\infty$. Let $u_{\vep}\in W^{s_1,p}_0(\Om)$, for $\vep\in (0,1)$, be the family of weak solution to problem \eqref{probM}  with $f_\e(x):=f_\e(x,u_\e)$ satisfying $|f_\e(x,t)|\leq C_0(1+|t|^{p^*_{s_1}-1})$, for all $t\in\mb R$ and a.e. $x\in\Om$, where $C_0$ is independent of $\e$. Assume that the sequence $\{ \|u_\vep\|_{W^{s_1,p}_0(\Om)} \}_\vep$ is bounded and $u_\vep\to u_0$ in $L^{p^*_{s_1}}(\Om)$, as $\vep\to 0$. 
 Then the sequence $\{ \|u_\vep\|_{L^\infty(\Om)} \}_\vep$ is also bounded.
\end{Corollary}

 \subsection{Interior regularity}
 Let $u:\mb R^N\to \mb R$ be a measurable function and $h\in\mb R^N$, then we define 
    \begin{align*}
     u_h(x)=u(x+h), \quad \de_h u(x)=u_h(x)-u(x), \quad \de^2_h u(x)= \de_h(\de_h u(x))= u_{2h}(x)+u(x)-2u_h(x).
    \end{align*} 
 For $1\le m<\infty$ and $u\in L^m(\mb R^N)$, we set
 \begin{align*}
  	[u]_{\mc B^{\ba,m}_{\infty}(\mb R^N)}:= \sup_{|h|>0} \bigg\| \frac{\de^2_h u}{|h|^\ba}\bigg\|_{L^m(\mb R^N)} \mbox{ for }  \ba\in (0,2).
 \end{align*}
 Now, we prove our improved interior H\"older regularity result for local weak solutions.\\
\textbf{Sketch of the proof of Theorem \ref{impintreg}}:  
We first consider the case $\sg\in (0,s_1)$. For  $qs_2<ps_1+2(1-s_1)$, we claim that,
  for every $4h_0<R\leq R_0 - 5h_0$, there holds: 
\begin{align}\label{eqitr}
	\sup_{0<|h|<h_0} \bigg\| \frac{\de^2_h u}{|h|^{s_1}}\bigg\|^{m+1}_{L^{m+1}(B_{R-4h_0})} 
	\leq C K_2(u,m) \Big( \sup_{0<|h|<h_0} \bigg\| \frac{\de^2_h u}{|h|^{s_1}}\bigg\|^m_{L^m(B_{R+4h_0})} +1 \Big),
\end{align}
where $m\geq p$, $h_0=R_0/10$, $C= C(N,h_0,p,m,s_1)>0$ (which depends inversely on $h_0$) and 
{\small\begin{align*}
	K_2(u,m,R_0):=1+T_{p-1}(u;R_0)^{p-1}+T_{q-1}(u;R_0)^{q-1}+ \| u \|^{\frac{m(p-1)}{p-2}}_{L^\infty(B_{R_0})}+\| u \|^{q-1}_{L^\infty(B_{R_0})}+\|f\|_{L^\infty(B_{R_0})}.
\end{align*}}
Indeed, for $2\leq q\leq p<\infty$, \eqref{eqitr} is proved in \cite[Proposition 3.9]{JDS2}. For the other case, we set $S_1=\{(x,y)\in B_R\times B_R \; : \; |x-y|\leq 1\}$ and $S_2=(B_R\times B_R)\setminus S_1$. Then,  the proof, in this case, runs similarly by noting the following (using the same notations of the Proposition),
\begin{align*}
 &\left(\iint_{S_1}+\iint_{S_2}\right) \frac{|u(x)-u(y)|^{q-2}}{|x-y|^{N+qs_2}} \big|\eta^\frac{p}{2}(x)-\eta^\frac{p}{2}(y)\big|^2 \frac{|\de_hu(x)|^{\ba+1}}{|h|^{1+\nu\ba}} dxdy \nonumber \\
 &\leq C \| u \|^{q-p+1}_{L^\infty(B_{R_0})}\iint_{S_1} \frac{|u(x)-u(y)|^{p-2}}{|x-y|^{N+ps_1+\al-2}}  \frac{|\de_hu(x)|^{\ba}}{|h|^{1+\nu\ba}} dxdy +C \| u \|^{q-1}_{L^\infty(B_{R_0})}\int_{B_R} \frac{|\de_hu(x)|^{\ba}}{|h|^{1+\nu\ba}} dx \nonumber\\
 &\leq CK_2(u) [u]^m_{W^{\frac{s_1(p-2-\e)}{p-2},m}(B_{R+h_0})} + C K_2(u) \Big(\int_{B_R}  \frac{|\de_hu(x)|^{\frac{\ba m}{m-p+2}}}{|h|^\frac{(1+\nu\ba)m}{m-p+2}}dx +1\Big)
\end{align*}
where we have used H\"older's and Young's inequality together with the fact that $qs_2\leq ps_1+\al$ with $\al<2(1-s_1)$ and $\e\in (0,\frac{2-\al}{s_1}-2)$. Thus, $\tl I_{11}(q)$ (hence $I_{11}(q)$) is estimated as similar to $\tl I_{11}(p)$. Set 
  \begin{align*}
 	s_1-\sg > \frac{N}{p+i_\infty}, \quad h_0=\frac{\bar R_0}{64 i_\infty} \quad\mbox{for some }i_\infty\in\mb N
  \end{align*}
 and define the following sequences
  \begin{align*}
	m_i=p+i, \quad R_i= \frac{7\bar R_0}{8}-4(2i+1)h_0 \quad\mbox{for all } i=0,\dots,i_\infty.
  \end{align*} 
We take $\psi\in C_c^\infty(B_{(5\bar R_0)/8})$ such that 
  \[ 0\leq \psi \leq 1, \quad \psi=1 \mbox{ in } B_{\bar R_0/2}, \quad |\na\psi|\mbox{ and } |\na^2\psi|\leq C. \]
 Using the discrete Leibniz rule on $\de^2_h$,  we obtain
 \begin{align*}
 	[u\psi]_{\mc B^{s_1,m_{i_\infty}}_{\infty}(\mb R^N)}
 	\leq C  \Big[ \sup_{0<|h|<h_0}\bigg\| \frac{\de^2_h u}{|h|^{s_1}}\bigg\|_{L^{m_{i_\infty}}(B_{(3\bar R_0)/4})} + \| u \|_{L^{m_{i_\infty}}(B_{(3\bar R_0)/4})} \Big].
 \end{align*}
 The first term of the right hand side on the above expression is estimated on account of \eqref{eqitr}. Therefore, 
 employing the embedding result of the Besov spaces into the H\"older spaces, we get that $u\in C^{0,\sg}_{\mathrm{loc}}(\Om)$, for all $\sg\in (0,s_1)$. For the case $2\leq p<q$ and $ps_1<qs_2$, we proceed exactly as above by interchanging the role of $(p,s_1)$ with $(q,s_2)$ and the corresponding spaces. In this case, $\tl I_{11}(p)$ is estimated as above by choosing $\e>0$ such that $\e<\frac{2}{s_2}-q+p-2$.\\
 The higher regularity result follows by using the above almost $s_1$ (or $s_2$)-H\"older continuity result and proceeding on the similar lines of the proof of \cite[Theorem 5.2]{brascoH} (with minor modification as in the proof above).
\QED

 \subsection{Boundary regularity and maximum principle}
 In this subsection, we prove the boundary behavior of the weak solutions. 
 For $\al,\rho>0$ and $\ka\ge 0$, we set 
{\small \begin{align*}
 	d_e(x)=\begin{cases}
 		d(x) &\mbox{if }x\in\Om,\\
 		-d(x) &\mbox{if }x\in(\Om^c)_{\rho},\\
 		-\rho &\mbox{otherwise},
 	\end{cases}
  \quad \ov w_{\rho}(x)=\begin{cases}
  	(d_e(x)+\ka^{1/\al})_+^\al \quad &\mbox{if }x\in\Om\cup(\Om^c)_\rho, \\
  	0 \qquad&\mbox{otherwise},
  \end{cases}
 \end{align*}}
 where $(\Om^c)_{\rho}:=\{x\in\Om^c : \mathrm{dist}(x,\pa\Om)<\rho \}$. \\
 \textbf{Sketch of the proof of Theorem \ref{bdryreg}}: We proceed as below.
\begin{itemize}
 \item[(a)] By flattening the boundary $\pa\Om$ and using suitable $C^{1,1}(\mb R^N,\mb R^N)$ diffeoemorphisms, we prove that:
  there exist $\ka_1,\varrho_1>0$ such that for all $\ka\in[0,\ka_1)$ and  $\varrho\in(0,\varrho_1)$,
  \begin{align*}
  	(-\De)_p^{s_1}\ov w_\rho \begin{cases}
  		\geq C_1 (d+\ka^{1/\al})^{-(ps_1-\al(p-1))} \quad\mbox{for all }\al\in (0,s_1), \\
  		= h \quad\mbox{for all }\al\in [s_1,1) \mbox{ with }\al\neq p's_1
  	\end{cases} 
  \quad\mbox{weakly in }\Om_{\varrho},
  \end{align*} 
  where $C_1>0$ is a constant and $h\in L^\infty(\Om_{\varrho_1})$ (both are independent of $\ka\in (0,1)$). 
  Further, for all $\ka>0$ and $\al\in (0,s_1)$, $\ov w_\rho \in \widetilde{W}^{s_1,p}(\Om_{\varrho_1})$, and for $k=0$, $\ov w_\rho \in \widetilde{W}^{s_1,p}(\Om_{\varrho_1})$, whenever $\al> s_1-1/p$.
  \item[(b)] For $\Ga>1$, $\max\{s_1-1/p,s_2-1/q\}<\al<s_1$ and $\varrho>0$ (sufficiently small), we have, weakly in $\Om_{\varrho}$,
    \begin{align*}
     (-\De)_p^{s_1}(\Ga d^\al)+(-\De)_q^{s_2}(\Ga d^{\al}) &\geq  C_5 \Ga^{p-1} d^{-(ps_1-\al(p-1))}-\Ga^{q-1}\|h\|_{L^\infty(\Om_{\varrho})} \\
     &\geq C_6 \Ga^{p-1}d^{-(ps_1-\al(p-1))}.
    \end{align*}
 Then, employing the weak comparison principle in $\Om_{\varrho}$, for suitable $\Ga$, we get that $ u \leq \Ga d^{\al}$ in  $\Om$. Subsequently, we perform a similar process for $-u$ also.
 \item[(c)] The proof of the H\"older continuity can be completed by taking into account Theorem \ref{impintreg} and the boundary behavior presented in Step (b).  \QED
\end{itemize}
\begin{Remark}\label{rem1opt}
 We remark that in Theorem \ref{bdryreg}, the choice of $\sg$ can be optimal (that is, $\sg=s_1$) for the case $s_1=s_2$ or $s_1>q's_2$. Indeed, for $s_1=s_2$, we can show that the barrier function as constructed in \cite[Lemma 4.3]{iann} satisfies $(-\De)_q^{s_1} w\geq 0$ weakly in $B_{r}(e_N)\setminus\ov{B_1}$. Thus, for appropriate choice of $\Ga>1$, the Step (b) above can be improved. Similar arguments apply to the case $s_1>q's_2$ with a careful reading of the proof of \cite[Lemma 3.12]{JDS2}.
\end{Remark}
 \textbf{Proof of Corollary \ref{cor1}}:  When $f(x):=f(x,u)$, on account of Theorem \ref{mainbound} and Remark \ref{rem2}, we observe that  
 $$|f(x,u)|\leq C_0 (1+|u|^{p^*_{s_1}-1}) \leq C_0 (1+\|u\|_{L^\infty(\Om)}^{p^*_{s_1}-1})=:K>0.$$ 
Thus, the required result, in this case, follows from Theorem \ref{bdryreg}. \QED
 Next we state our strong maximum principle. The proof is contained in \cite{JDSjga} and done by proving that continuous weak super-solutions are viscosity super-solutions.
 \begin{Theorem}\label{strongmax}
 	Suppose that $1<q\leq p<\infty$. Let $g\in C(\mb R)\cap BV_{\rm loc}(\mb R)$ and let $u\in W^{s_1,p}_0(\Om)\cap C(\ov\Om)$ be such that 
 	\begin{align*}
 		(-\De)_p^{s_1} u+ (-\De)_q^{s_2} u +g(u)\geq g(0) \quad\mbox{weakly in }\Om.
 	\end{align*}
 	Further, assume that $u\not\equiv 0$ with $u\geq 0$ in $\Om$.  Then, there exists $c_1>0$ such that $u\geq c_1 {\rm dist}(\cdot,\pa\Om)^{s_1}$ in $\Om$.
 \end{Theorem}
    
\textbf{Sketch of the proof of Theorem \ref{strongcompreg}}:
 By continuity and the fact that $u\not\equiv v$, we can find $x_0\in\Om$, $\rho,\e>0$ such that $B_\rho(x_0)\subset\Om$ and 
	\begin{align}\label{eq190}
		\sup_{B_\rho(x_0)}  v < \inf_{B_\rho (x_0)} u -\e/2.
	\end{align}
  For $\Ga>1$ and for all $x\in\mb R^N$,  we define 
	{\small\begin{align*}
		w_\Ga(x)=\begin{cases}
			\Ga v(x) &\mbox{if }x\in B_{\rho/2}^c(x_0)\\
			u(x) &\mbox{if }x\in B_{\rho/2}(x_0).
		\end{cases}
	\end{align*}}
Taking into account the nonlocal super-position principle \cite[Lemma 2.5]{DDS}, we have, weakly in $\Om\setminus B_\rho(x_0)$,
 {\small\begin{align*}
 	(-\De)_p^{s_1}w_\Ga +(-\De)_q^{s_2} w_\Ga 
 	&\leq (-\De)_p^{s_1}u +(-\De)_q^{s_2} u + (\Ga^{p-1}-1)K+ (\Ga^{p-1}- \Ga^{q-1})K_1-C_1 \e^{p-1}-C_2 \e^{q-1}.
 \end{align*}}
We can choose $\Ga>1$ (close to $1$) to employ the weak comparison principle (\cite[Proposition 2.6]{DDS}), consequently, we get $w_\Ga\leq u$ in $\Om$. Hence, using \eqref{eq190} and Theorem \ref{strongmax}, we obtain $u\geq \Ga v>v$ in $\Om$, and  $\frac{u-v}{d^{s_1}}\geq \frac{(\Ga-1)v}{d^{s_1}}\geq C>0$ in $\Om$. \QED

\section{Applications}
We consider the problem \eqref{probM} with the choice $f(x):=f(x,u)$, where $f:\Om\times\mb R\to\mb R$ is a Carath\'eodory function satisfying the following:
\begin{itemize}
 \item[(A1)] $|f(x,t)|\leq C_0 (1+|t|^{r-1})$, for a.a. $x\in\Om$ and all $t\in\mb R$, where $C_0>0$ is a constant and $r\in (1,p^*_{s_1}]$.
 \item[(A2)] For a.a. $x\in\Om$, $f(x,t)t\leq 0$, for all $t\in [-\varsigma,\varsigma]$ ($\varsigma>0$) and $f(x,t)t\geq -c_1t^p$ ($c_1>0$) for all $t\in\mb R$.
 \item[(A3)] For $F(x,t):=\int_{0}^{t}f(x,\tau)d\tau$, $\lim_{|t|\to\infty}\frac{F(x,t)}{|t|^p}=\infty$  uniformly for a.a. $x\in\Om$.
 \item[(A4)] Let $r\in (p,p^*_{s_1})$, there exists $\nu\in \big( (r-p)\max\{N/(ps_1), 1\}, p^*_{s_1} \big)$ such that 
 \begin{align*}
 	\lim_{|t|\to\infty}\frac{f(x,t)t-pF(x,t)}{|t|^\nu}>0 \quad\mbox{uniformly a.e. }x\in\Om.
 \end{align*}
\end{itemize}
One example for $f$ satisfying (A1)-(A4) is given by $f(x,u)=-c_1|t|^{p-2}t+|t|^{r-2}t$.
 The Euler functional $\mc {J}: W^{s_1,p}_0(\Om) \ra\mb R$ associated to  problem \eqref{probM} is given by
 \[\mc {J}(u)= \frac{1}{p}\|u\|_{W^{s_1,p}_0(\Om)}^{p}+ \frac{1}{q}\|u\|_{W^{s_2,q}_0(\Om)}^{q} -\int_{\Om}F(x,u)dx. \]
First we prove the following Sobolev versus H\"older minimizer result.
\begin{Theorem}\label{sobvshol}
 Suppose that $2\leq q\leq p<\infty$ and (A1) holds. Let $u_0 \in W^{s,p}_0(\Omega)$,
  then for all $\al\in (0,s_1)$, the following are equivalent
 \begin{enumerate}
	\item[(i)] there exists $\sigma>0$ such that $\mc J(u_0+v) \geq \mc J(u_0)$ for all $v\in W^{s_1,p}_0(\Om)$, $\| v \|_{W^{s_1,p}_0(\Om)}\leq \sigma$,
	\item[(ii)] there exists $\omega>0$ such that $\mc J(u_0+v) \geq \mc J(u_0)$ for all $v\in W^{s_1,p}_0(\Om)\cap C^{0,\al}(\ov\Om)$ with $\| v \|_{C^{\al}(\ov\Om)}\leq \omega$.
 \end{enumerate}
\end{Theorem}	
\begin{proof}
From $(ii)$ and the density argument, we get that $\langle \mc J^\prime(u_0),\phi \rangle = 0$ for all $\phi\in W^{s_1,p}_0(\Om)$. Consequently, Theorem \ref{bdryreg} implies that $u_0\in C^{0,\al}(\ov\Om)$. To prove $(i)$, on the contrary assume that there exists $\tilde u_n\in W^{s_1,p}_0(\Om)$ such that $\tilde u_n\to u_0$ in $W^{s_1,p}_0(\Om)$ and $\mc J(\tilde u_n)<\mc J(u_0)$ for all $n\in\mb N$. Set 
	\[ \mc K(v)=\frac{1}{p^*_{s_1}}\int_{\Om}|v|^{p^*_{s_1}}, \ \vep_n:= \mc K(\tilde{u_n}-u_0) \ \mbox{and } S_n:= \{ u\in W^{s_1,p}_0(\Om) \ : \ \mc K (u-u_0) \leq \vep_n  \}. \]
 By the continuous embedding $W^{s_1,p}_0(\Om) \hookrightarrow L^{p^*_{s_1}}(\Om)$, we see that $\vep_n\to 0$ and hence $S_n$ is a closed convex subset of $W^{s_1,p}_0(\Om)$. Next, for all $t\in\mb R$ and $k>0$, set $[t]_k=\mathrm{sign}(t)\min\{ |t|,k\}$ and $f_k(x,t):=f(x,[t]_k)$ with $F_k(x,t):=\int_{0}^{t}f_k(x,\tau)d\tau$. Then, on account of the Lebesgue dominated convergence theorem, for fixed $n\in\mb N$ and $\sg_n\in (0, \mc J(u_0)-\mc J(\tilde u_n))$, there exists $k_n >\|u_0\|_{L^\infty(\Om)}+1$ such that 
	\begin{align*}
		\Big|\int_{\Om} F_n(x,\tilde u_n)dx -\int_{\Om} F(x,\tilde u_n)dx\Big| <\sg_n,
	\end{align*}
	where $F_n=F_{k_n}$. Furthermore, we define 
	\begin{align*}
		\mc J_n(u)= \frac{1}{p}\|u\|_{W^{s_1,p}_0(\Om)}^{p}+ \frac{1}{q}\|u\|_{W^{s_2,q}_0(\Om)}^{q}  -\int_{\Om}F_n(x,u) ~dx.
	\end{align*}
 From the structure of the function $F_n$, it is clear that there exists a minimizer $u_n\in S_n$ for $\mc J_n$.
	Moreover, by the choice of $\sg_n$ and $k_n$, we see that 
	\begin{align}\label{eq37}
		\mc J_n(u_n) \leq \mc J_n(\tilde u_n)\leq \mc J(\tilde u_n)+\sg_n < \mc J(u_0)=\mc J_n(u_0).
	\end{align}
It is clear that $\mc J_n$ is G\v{a}teaux differentiable at $u_n$. Therefore, there exists $\mu_n\leq 0$ such that 
	\begin{equation*}
	 (\mc P_n)\left\{\begin{array}{rllll}
	   (-\Delta)^{s_1}_{p}u_n+(-\Delta)^{s_2}_{q}u_n & = f_n(x,u_n)+\mu_n |u_n-u_0|^{p^*_{s_1}-2}(u_n-u_0) \; \text{ in } \Om, \\ u_n&=0 \quad \text{ in } \mathbb{R}^N\setminus \Om.
		\end{array}
		\right.
	\end{equation*}
 If $\inf_n \mu_n:=l >-\infty$, then from the fact that $u_0\in L^\infty(\Om)$, we have
	\begin{align*}
		|f_n(x,u_n)+\mu_n |u_n-u_0|^{p^*_{s_1}-2}(u_n-u_0)| \leq C (1+|u_n|^{p^*_{s_1}-1}).
	\end{align*}
If $\inf_n \mu_n:= -\infty$, there exists $M>0$ (independent of $n$) such that 
	\[ f_n(x,t)+\mu_n |t-u_0(x)|^{p^*_{s_1}-2}(t- u_0(x))<0 \quad\mbox{for a.a. } x\in\Om, \mbox{ and all }t\in(M,\infty). \]
 This implies that $u_n\leq M$ for all $n\in\mb N$. Further, since $\mc J^\prime(u_0)=0$, we take $w=|u_n-u_0|^{\ka-1}(u_n-u_0)$ as a test function and using \cite[Lemma 2.3]{iannsob} (consequently, the difference of terms involving $A_q$, below, is non-negative), we have 
	{\small\begin{align*}
		 (C\ka^{p-1})^{-1} \| (u_n-u_0)^\frac{p-1+\ka}{p}\|^p_{W^{s_1,p}_0(\Om)} 
			&\leq A_p(u_n,[u_n-u_0]^\ka)-A_p(u_0,[u_n-u_0]^\ka) \\
			&\quad+ A_q(u_n,[u_n-u_0]^\ka) -A_q(u_n,[u_n-u_0]^\ka) \\
			&=\int_{\Om} (f_n(x,u_n)-f(x,u_0))[u_n-u_0]^\ka +\mu_n\int_{\Om} |u_n-u_0|^{p^*_{s_1}-1+\ka}. 
	\end{align*}}
 Noting the uniform bound $\|u_n \|_{L^\infty(\Om)} \leq M$ and using H\"older's inequality, and subsequently passing to the limit $\ka\to\infty$, we obtain
	\begin{align*}
		-\mu_n \| u_n-u_0 \|^{p^*_{s_1}-1}_{L^{\infty}(\Om)} \leq C
	\end{align*}
 where $C>0$ is a constant independent of $n$. Thus, in all the cases, we obtain
	\[ |f_n(x,u_n)+\mu_n |u_n-u_0|^{p^*_{s_1}-2}(u_n-u_0)| \leq C (1+|u_n|^{p^*_{s_1}-1}) \quad\mbox{for all }n\in\mb N. \]    
 Moreover, from the construction of $u_n$,  it is clear that $\{ \|u_n\|_{W^{s_1,p}_0(\Om)}\}_n$ remains bounded. Then, applying Corollary \ref{bound}, we have $\| u_n \|_{L^\infty(\Om)}\leq C$, where $C>0$ is a constant independent of $n$. Consequently, from Corollary \ref{cor1}, we deduce that $\| u_n \|_{C^{\al}(\ov\Om)}\leq C$, for some positive constant $C$ independent of $n$, and all $\al\in (0,s_1)$. Therefore, by Arzela-Ascoli's theorem,  $u_n\to u_0$ in $C^{0,\al}(\ov\Om)$, for all $\al<s_1$. Thus, for sufficiently large $n$, we have $\| u_n-u_0 \|_{C^{\al}(\ov\Om)}\leq \omega$, and since $u_n$ is uniformly bounded in $L^\infty(\Om)$, $\mc J_n (u_n)=\mc J(u_n)$, for sufficiently large $n$. 
 This along with \eqref{eq37} contradicts the fact that $u_0$ is a minimizer for $\mc J$ in $W^{s_1,p}_0(\Om)\cap C^{0,\al}(\ov\Om)$.\\
 The proof of the other implication is standard.\QED
\end{proof}	
\begin{Theorem}
 Suppose that $2\leq q\leq p<\infty$. Then, there exist at least three non-trivial solutions $u\in W^{s_1,p}_0(\Om)\cap C^{0,\al}(\ov\Om)$, for all $\al\in(0,s_1)$, to problem \eqref{probM} with $f(x):=f(x,u)$ satisfying (A1)-(A4). Moreover,   if $u$ is non-negative and $s_1\neq q's_2$, then $u\geq c d^{s_1}$ in $\Om$.
\end{Theorem}
\textbf{Sketch of the Proof}: We consider the truncation of the nonlinear term as $f_\pm(x,t)=f(x,\pm t^\pm)$ with $F_\pm(x,t):=\int_{0}^{t}f_\pm(x,\tau)d\tau$ and the corresponding Euler functionals as $\mc J_\pm$. By using Theorem \ref{sobvshol}, we can prove that $0$ is a local minimizer for $\mc J_+$ in $W^{s_1,p}_0(\Om)$ topology and it satisfies the mountain pass geometry. Thus, we obtain a positive solution $u_+$ to problem \eqref{probM}. Similar procedure yields a negative solution $u_-$. Subsequently, by using  topological tools (such as critical groups and Morse theory, see \cite{duzgun} for the linear operator case), we establish a third solution of undetermined sign nature. \QED 


\selectlanguage{English}

\end{document}